\documentclass[11pt,american]{amsart}
\usepackage{amsmath,amsfonts,amssymb,amscd,amsthm,amsbsy,bbm,epsf,calc,graphicx,esint,hyperref}
\usepackage{color}
\usepackage{datetime}
\usepackage{mathtools}

\numberwithin{equation}{section}

\newcounter{hours}\newcounter{minutes}

\theoremstyle{plain}
\newtheorem{theorem}{Theorem}[section]

\newtheorem{lemma}[theorem]{Lemma}

\theoremstyle{definition}                  
\newtheorem{remark}[theorem]{Remark}

\title{Regularization estimates of the Landau-Coulomb diffusion}

\begin{document}
	
\author{Rene Cabrera, Maria Gualdani, and Nestor Guillen}	
\thanks{ RC  is partially supported by the NSF DMS RTG 1840314. MPG is supported by NSF DMS 2019335 and would like to
thank the Isaac Newton Institute (Cambridge, UK) and NCTS Mathematics Division Taipei for their kind hospitality. NG is supported by NSF DMS 214423.} 
\address{The University of Texas at Austin,
Mathematics Department, 
2515 Speedway Stop C1200,
Austin, Texas 78712-1202.}
\email{gualdani@math.utexas.edu }
\address{The University of Texas at Austin,
Mathematics Department ,
2515 Speedway Stop C1200,
Austin, Texas 78712-1202.}
\email{rene.cabrera@math.utexas.edu}
\address{Texas State University,
Mathematics Department, 
Pickard St., San Marcos, TX 78666.}
\email{nestor@txstate.edu}

\date{\today}

\begin{abstract}

The Landau-Coulomb equation is an important model in plasma physics featuring both nonlinear diffusion and reaction terms. In this manuscript we focus on the diffusion operator within the equation by dropping the potentially nefarious reaction term altogether. We show that the diffusion operator in the Landau-Coulomb equation provides a much stronger $L^1 \to L^\infty$ rate of regularization than its linear counterpart, the Laplace operator.  The result is made possible by a nonlinear functional inequality of Gressman, Krieger, and Strain together with a De Giorgi iteration. This stronger regularization rate illustrates the importance of the nonlinear nature of the diffusion in the analysis of the Landau equation and raises the question of determining whether this rate also happens for the Landau-Coulomb equation itself.

\end{abstract}

\maketitle

\baselineskip=14pt
\pagestyle{headings}		

\markboth{Regularization estimates of the Landau-Coulomb diffusion}{R. Cabrera, M. Gualdani, N. Guillen}

\section{Introduction and main result}\label{Introduction}
 
We consider the nonlinear quadratic equation 
\begin{align}\label{P1}
u_t = \textrm{div}(A[u]\nabla u) ,\quad x\in \mathbb{R}{^d}, \quad t>0,
\end{align}
where $A[u]$ is a $d\times d$ matrix defined as 
$$
A[u](x,t) := c_d \int_{\mathbb{R}{^d} }\frac{\mathbb{P}(x-y)}{|x-y|^{d-2}}u(y,t)\;dy,
$$
and $\mathbb{P}(z)$ is the projection matrix over the space perpendicular to $z$, defined as 
$$
\mathbb{P}(z) :=  \mathbb{Id} - \frac{z \otimes z}{|z|^2},\quad z\in \mathbb{R}^{d}\setminus\{0\},\quad \text{and}\quad \mathbb{Id}\quad \text{identity matrix}
$$ 
Hereafter the dimension $d$ is greater than or equal to $3$. Equation (\ref{P1}) represents the ``diffusive'' part of the homogeneous Landau-Coulomb equation, in divergence form, which reads as 
\begin{align}\label{P2}
u_t = \textrm{div}(A[u]\nabla u - u \;\textrm{div}A[u]) ,
\end{align}
or alternatively in the non-divergence form
\begin{align}\label{P2 nondivergence}
u_t = \textrm{Tr}(A[u] D^2 u ) + u^2. 
\end{align}
The term in (\ref{P2}) that is missing in (\ref{P1}), namely $ -\textrm{div}(u \;\textrm{div}A[u])$,  had long been known to be the term that fights against the the regularization effect from the diffusive term. Accordingly, a complete analysis of the regularization rate for (\ref{P1}) allows for understanding the smoothing effects of this term on its own right, and might provide clues on the regularization rate for the homogeneous Landau-Coulomb equation (\ref{P2}). 

The homogeneous Landau equation (\ref{P2}) is a fundamental equation in kinetic theory, playing the role of the Boltzmann equation when the Coulomb force (as in the case of a plasma) is involved (see for instance the classical reference by Villani \cite{Villani-Handbook02} for a comprehensive introduction). The mathematical analysis of (\ref{P2}) has received considerable attention in the past two decades, and we highlight here (organized by topic) some of the contributions to the study of this equation. The collective investigation for (\ref{P2}) of the past years has advanced forward the knowledge in several directions:  (i) The existence and uniqueness of smooth solutions for short times have been extensively explored. Notably, Golding and Loher's work \cite{GoLo23}  has established the most comprehensive result for initial data in $L^p(\mathbb{R}^3)$ with $p>\frac{3}{2}$.  (ii) Global existence and uniqueness of smooth solutions with initial data close to equilibrium has been addressed across various contexts.  For initial data small in Sobolev spaces, contributions can be found in \cite{Guo02, CM2017verysoft, Desvillettes-He-Jiang-2021} and related references. The first work that considers initial data in $L^\infty$ is the one by  Kim, Guo and Hwang  \cite{KimGuoHwang2016}. Golding, Gualdani and Loher in \cite{GoGuLo23} encompassed the problem in all $L^p$ spaces with $p\ge \frac{3}{2}$. Very recently the case $p\ge1$ was solved in \cite{DGL24} and \cite{Ji24}. (iii) Conditional regularity. This line of research concerns the investigations of conditions that guarantee global well-posedness of solutions for arbitrarily large times. Silvestre \cite{Silvestre2015} and Gualdani, Guillen \cite{GuGu16} showed that if the function $u(x,t) \in  L^p(\mathbb{R}^3)$ with $p>\frac{3}{2}$ uniformly in time, then it is smooth. Recently, Alonso, Bagland, Desvillettes, and Lods \cite{ABDL23} showed that if $u(x,t) \in  L^q(0,T,L^p(\mathbb{R}^3))$ for a certain range of $q$ and $p$, then it is automatically in $L^p(\mathbb{R}^3)$ with $p>\frac{3}{2}$ and therefore smooth. Regarding conditional uniqueness, Fournier in \cite{Fournier2010} showed that solutions which have $L^\infty$ norm integrable in time are unique. Chern and Gualdani \cite{ChGu20} showed that uniqueness holds  in the class of high  integrable  functions.    (iv) Partial regularity. This line of research for the Landau equation started with Golse, Gualdani, Imbert and Vasseur \cite{GGIV2019partial}. In this work it is shown that, if singularities occur, they are concentrated in a time interval that has Hausdorff measure at most $\frac{1}{2}$. Most recently, Golse, Imbert and Vasseur showed that the spatial and temporal domain for singularities to happen has Hausdorff measure $1 + \frac{5}{2}$ \cite{GIV2023}. (v) Study of modified models that pertain the same difficulties of the Landau equation but seem analytically more tractable. This line of research started with the work of Gressmann, Krieger and Strain \cite{KriStr2012,GreKriStr2012} and their analysis of an isotropic version of (\ref{P2}), 
\begin{align}\label{P2_iso}
u_t = a[u] \Delta u + \alpha u^2, \quad \alpha >0. 
\end{align}
In \cite{GreKriStr2012, KriStr2012} they show that (\ref{P2_iso}) is globally well-posed if initial data are radially symmetric and monotonically decreasing and $\alpha \in (0,\frac{74}{75})$. Later, Gualdani and Guillen \cite{GuGu15} proved global well-posedness for $\alpha =1$ also in the case when initial data are radially symmetric and monotonically decreasing. These works where the first to establish that, unlike what happens in the semilinear heat or porous media equations, the nonlinear diffusion $a[u] \Delta u$ is strong enough to overcome the reaction $u^2$. Later, Gualdani and Guillen  \cite{GG21} showed the the isotropic Landau equation with less singular potentials ($\gamma \in (-2.5, -2]$) is also globally well-posed. 

These findings lead us to the motivation behind the present manuscript. The proofs in groups (i), (iii) and (iv) primarily rely on the ellipticity estimates provided by the lower bound of $A[u]$. Specifically, if the function $u$ has mass, second moment and entropy bounded, the matrix $A[u]$ is uniformly bounded from below by
$$
A[u] \ge \frac{c}{1+ |x|^d} \mathbb{Id},  \quad x \in \mathbb{R}{^d},
$$
where $c$ only depends on mass, second moment and entropy of $u$.  While a weighted Laplacian operator is analytically more tractable then the full nonlinear nonlocal diffusion $\textrm{div}(A[u]\nabla u)$, one might argue that by using the lower bound on $A[u]$ we discard an important element that could actually prevent singularities following the intuition that when $u$ is big so is the diffusion coefficient $A[u]$ and this strength could prevent formation of singularities. However, this intuition has been very difficult to apply in practice. Interestingly, however, is the fact that all the global-well-posedness results for general data mentioned above use the full power of the diffusion operator. These include the results in  \cite{GreKriStr2012, KriStr2012}, which are a consequence of a novel weighted Poincare inequality 
$$
\int_{\mathbb{R}^d} u^{p+1}\;dx \le \left( \frac{p+1}{p}\right)^2 \int_{\mathbb{R}^d} A[u] |\nabla u^{\frac{p}{2}}|^2\;dx,
$$ 
the ones in \cite{GuGu15}, which use a geometric argument in which the coefficient $a[u]$ plays a pivotal role, and lastly, the ones in   \cite{GG21}, proven via new weighted Hardy  inequalities  of the form
$$
\int _{\mathbb{R}^d} (u\ast |x|^\gamma)u^p\;dx \le c_{d,\gamma,p} \int_{\mathbb{R}^d} (u\ast |x|^{\gamma+2})|\nabla u^{\frac{p}{2}}|^2\;dx ,\quad \gamma>-d.
$$
Lastly, it was already noted in \cite{GuGu16}  that conditional regularity result for (\ref{P2}) shows a rate of regularization much stronger than what is usually expected for regular parabolic equations. \\

In this manuscript we provide  a new and precise quantification of the regularization power of the Landau diffusion operator. Notably, this regularization exhibits a significantly faster rate than that achieved by the Laplacian operator. 
\begin{theorem}\label{main thm}
    Let $u(t,x)$ be a solution to (\ref{P1}) with initial data $0\le u_{in}$ that belongs to $L^{1}_{m}(\mathbb{R}^{d})$ for some $m>3d(d-2)$ and $d\ge 3$. Then the following estimate holds for any small $\varepsilon>0$ and any times $0<t<T$, 
    \begin{align}\label{main result}
       \|u\|_{L^{\infty}(t, \;T,\; L^{\infty}(\mathbb{R}^{d}))}\leq \frac{c_{\varepsilon}}{t^{1+\varepsilon}},
    \end{align}
Here $c_\varepsilon>0$ is a constant depending only on $\varepsilon$, $d$, and the $L^{1}_{m}$-norm of the initial data $u_{in}$.    
\end{theorem}
Our best bound for the constant $c_\varepsilon$ in (\ref{main result}) is one that goes to infinity as $\varepsilon \to 0+$, so we cannot obtain a rate exactly with $\varepsilon = 0$. Whether an estimate with $\varepsilon=0$ holds is an interesting question but one that might not be possible with our approach. We remind the reader that in contrast to this $t^{-1-\varepsilon}$ regularization rate, the heat equation in $\mathbb{R}^d$ has a (sharp) regularization rate for the $L^\infty$ norm of $t^{-d/2}$. This bound when $d\geq 3$ is much larger than the $t^{-1-\varepsilon}$ rate ($\varepsilon$ small) when $t$ is small, indicating that the $L^\infty$ norm of $f$ regularizes (that is, comes down from $+\infty$) much faster for (\ref{P1}).

In a preprint following the completion of this manuscript \cite{GuSi2023}, the third author and Luis Silvestre prove that the Fisher information for solutions of the Landau-Coulomb equation \eqref{P1} does not increase, and thus solutions starting from smooth data with fast decay must remain smooth for all times. Also recently \cite{Chen2023}, Chen proved that blow up occurs for the equation obtained by increasing the coefficient in the reaction term in \eqref{P2 nondivergence} by an arbitrary positive amount. These results together with Theorem \ref{main thm} raise the interesting and worthwhile question of whether the fast regularization rate in Theorem \ref{main thm} continues to hold for solutions of the full Landau-Coulomb equation (\ref{P2}). Similarly it would be worthwhile to investigate similar regularization rates for the Landau equation for very soft potentials.

Lastly, we clarify some notation. In what follows we will denote by $L^{1}_{m}(\mathbb{R}^{d})$ the space of all $L^1(\mathbb{R}^{d})$ functions such that 
$$
\int_{\mathbb{R}{^d}} |f|(1+ |x|^2)^{m/2}\;dx < +\infty.
$$

The proof of Theorem \ref{main thm} can be divided in two stages. In the first one we show a $L^1 \to L^p$ gain of integrability for $u$, solution to (\ref{P1}). This is possible thanks to the nonlinear Sobolev inequality involving $A[u]$ discovered by Gressmann, Krieger and Strain \cite{GreKriStr2012} and which captures the strong regularization effects of the Dirichlet form associated to $A[u]$ -- this is explained in Section \ref{sec:Lp gain}. The second stage is a proof in the style of De Giorgi-Nash-Moser theory, which we use to obtain the $L^p \to L^\infty$ part of the estimate -- this is the content of Section \ref{sec:De Giorgi}. The combination of these two steps yields (\ref{main result}). Some preparatory lemmas are discussed in the next Section \ref{sec:technical lemmas}.

\section{Some technical lemmas}\label{sec:technical lemmas}

We first recall well-known results on the bounds of the diffusion matrix $A[u]$. For the proof of the following lemma, see for example \cite[Lemma 3.1]{bedrossian2022non}.
\begin{lemma}\label{prop 1}
There exist positive constants $C_{0}$ and $c_{0}$ depending on the dimension $d\geq 3$ such that 
\begin{equation*}
\|A[u]\|_{L^{\infty}(\mathbb{R}^{d})}\; \leq \; C_{0}\|u\|^{\frac{p(d-2)}{d(p-1)}}_{L^{p}(\mathbb{R}^{d})}\|u\|^{\frac{2p-d}{d(p-1)}}_{L^{1}(\mathbb{R}^{d})},\quad p>\frac{d}{2},
\end{equation*}
and
\begin{align*}
    \|{div} A[u]\|_{L^{\infty}(\mathbb{R}^{d})}\; \leq \; c_{0}\|u\|^{\frac{p(d-1)}{d(p-1)}}_{L^{p}(\mathbb{R}^{d})}\|u\|_{L^{1}(\mathbb{R}^{d})}^{\frac{p-d}{d(p-1)}},\quad p>d.
\end{align*}
\end{lemma}

We will also use the following weighted Sobolev inequality: for $f$ smooth enough and any $1\leq s \leq \frac{2d}{d-2}$ we have
\begin{align}\label{sawyer wheeden bound}
\left(\int_{\mathbb{R}^{d}}|f|^{\frac{2d}{d-2}}\langle x \rangle^{-3d}\;dx\right)^{\frac{d-2}{d}}\leq c_{1}\int_{\mathbb{R}^{d}}|\nabla f|^{2}\langle x \rangle^{-d}\;dx+c_{2}\left(\int_{\mathbb{R}^{d}}|f|^{s}\;dx\right)^{2/s}.
\end{align}
Here, $\langle x \rangle := (1+|x|^2)^{1/2}$. The derivation of (\ref{sawyer wheeden bound}) follows the steps in \cite{GoGuLo23}; where the authors prove it for the case $d=3$. Furthermore, the constants $c_{1}$ and $c_{2}$ depend only on the dimension $d$. 
We apply (\ref{sawyer wheeden bound}) in order to prove the following interpolation inequality: 

\begin{lemma}\label{Holde-interpolation ineq}
Let $p>1$ and $q$ such that $p+\frac{2}{d}<q< p\left(1+ \frac{2}{d}\right)$. Let $m$ be defined as 
$$
m:=\frac{3d(d-2)(p-1)}{(d+2)p-dq}.
$$ 
For any $g$ smooth function the following bound holds:
    \begin{align}\label{lemma4.4 ineq} 
    \|g\|^{q}_{L^{q}(\mathbb{R}^{d})}\leq C\left(\|\langle \cdot\rangle^{-\frac{d}{2}}\nabla g^{\frac{p}{2}}\|^{2}_{L^{2}(\mathbb{R}^{d})}+\|g\|_{L^{p}(\mathbb{R}^{d})}^{p}\right)\|g\|_{L^{p}(\mathbb{R}^{d})}^{p\left(\frac{q-p-\frac{2}{d}}{p-1}\right)}\|g\langle\cdot\rangle^{m}\|_{L^{1}(\mathbb{R}^{d})}^{\frac{(d+2)p-dq}{d(p-1)}}. 
    \end{align}
\end{lemma}

\begin{proof}
    We first establish the following interpolation inequality
\begin{align}\label{star ineq}
    \|g\|^{q}_{L^{q}}\leq \|\langle \cdot\rangle^{-\frac{3(d-2)}{p}}\;g\|_{L^{\frac{dp}{d-2}}}^{p}\;\|g\|_{L^{p}}^{p\left(\frac{q-p-\frac{2}{d}}{p-1}\right)}\|g\;\langle\cdot\rangle^{m}\|_{L^{1}}^{\frac{(d+2)p-dq}{d(p-1)}},
\end{align}
that holds for any $p+\frac{2}{d}<q< p\left(1+ \frac{2}{d}\right)$ and $m=\frac{3d(d-2)(p-1)}{(d+2)p-dq}$. 
The lemma follows easily once we prove (\ref{star ineq}): use (\ref{sawyer wheeden bound}) with $f=g^{\frac{p}{2}}$ and $s=2$ to bound the first term on the right hand side of (\ref{star ineq}) and get 
\begin{align*}
    \|g\|^{q}_{L^{q}}\leq C\left(\|\langle \cdot\rangle^{-\frac{d}{2}}\nabla g^{\frac{p}{2}}\|^{2}_{L^{2}}+\|g\|_{L^{p}}^{p}\right)\|g\|_{L^{p}}^{p\left(\frac{q-p-\frac{2}{d}}{p-1}\right)}\|g\langle\cdot\rangle^{m}\|_{L^{1}}^{\frac{(d+2)p-dq}{d(p-1)}}. 
\end{align*}

Next, to show (\ref{star ineq}) we start with a weighted interpolation 
\begin{align*}
    \|g\|^{q}_{L^{q}(\mathbb{R}^{d})}
\leq\left(\int_{\mathbb{R}^{d}}g^{\frac{dp}{d-2}}\langle x \rangle^{-\frac{dp\alpha}{\theta q(d-2)}}\;dx\right)^{\frac{\theta q(d-2)}{dp}}\left(\int_{\mathbb{R}^{d}}g^{r}\langle x \rangle^{\frac{r\alpha}{(1-\theta)q}}\;dx\right)^{\frac{(1-\theta)q}{r}},
\end{align*}
with $\alpha$, $\theta$ and $r$ satisfing
\begin{align*}
     \left\{ \begin{array}{rcl}
\frac{\theta q(d-2)}{dp}+\frac{(1-\theta)q}{r}=1,& \\ \frac{\theta q(d-2)}{dp}=\frac{d-2}{d}, & \\ 
\frac{dp\alpha}{\theta q(d-2)}=3d.
\end{array}\right.
\end{align*}
The above system has solutions  $\alpha=3(d-2)$, $\frac{(1-\theta)q}{r}=\frac{2}{d}$, $r=\frac{d}{2}(q-p)$, and $\theta=\frac{p}{q}$, which yield 
\begin{align*}
    \int g^{q}\;dx&\leq \left(\int g^{\frac{dp}{d-2}}\langle \cdot\rangle^{-3d}\;dx\right)^{\frac{d-2}{d}}\left(\int g^{\frac{d}{2}(q-p)}\langle \cdot \rangle^{\frac{3d(d-2)}{2}}\;dx\right)^{\frac{2}{d}}.
\end{align*}
Let us focus on the last term: once more, use H\"older's inequality and get
\begin{align}\label{split estimate}
    \begin{split}
    \left(\int g^{\frac{d}{2}(q-p)}\langle\cdot\rangle^{\frac{3d(d-2)}{2}}dx\right)^{\frac{2}{d}}&=\left(\int g^{\alpha}g^{\frac{d}{2}(q-p)-\alpha}\;\langle\cdot\rangle^{\frac{3d(d-2)}{2}}dx\right)^{\frac{2}{d}}\\
    &\leq \left[\left(\int g^{p}dx\right)^{\frac{\alpha}{p}}\left(\int g^{(\frac{d}{2}(q-p)-\alpha)\beta}\langle\cdot\rangle^{\frac{3d(d-2)}{2}\beta}dx\right)^{\frac{1}{\beta}}\right]^{\frac{2}{d}},
    \end{split}
\end{align}
where $\beta:=\frac{p}{p-\alpha}$. We choose $\alpha$ such that $(\frac{d}{2}(q-p)-\alpha)\beta=1$, which implies 
\begin{align*}
    \alpha=\left(\frac{d}{2}(q-p)-1\right)\frac{p}{p-1}.
\end{align*}
Note that $\alpha>0$; hence $\beta>1$, requires $q> p+\frac{2}{d}$. Since 
\begin{align*}
    \beta=\frac{2(p-1)}{(d+2)p-dq},
\end{align*}
 and as $p/\alpha$ has to be strictly greater than $1$, we require  $q<\left(\frac{d+2}{d}\right)p$. 
Substitution of $\alpha$ and $\beta$  in (\ref{split estimate}) yields
\begin{align}\label{a third of Holder estimate}
    \left(\int g^{\frac{d}{2}(q-p)}\langle\cdot\rangle^{\frac{3d(d-2)}{2}}dx\right)^{\frac{2}{d}}&\leq \left(\int g^{p}\;dx\right)^{\frac{q-p-\frac{2}{d}}{p-1}}\left(\int g\langle x \rangle^{m}\right)^{\frac{(d+2)p-dq}{d(p-1)}},
\end{align}
with $m=\frac{3d(d-2)(p-1)}{(d+2)p-dq}$. This proves (\ref{star ineq}) and finishes the proof.
\end{proof}

\begin{remark}\label{remark on constraint m}
    The condition $p+\frac{2}{d}<q<\frac{d+2}{d}p$ indicates that $m:=\frac{3d(d-2)(p-1)}{(d+2)p-dq}$ is such that $m>\frac{3d(d-2)}{2}$.
\end{remark}

\section{$L^1 \to L^{p}$ gain of integrability}\label{sec:Lp gain}

The next theorem shows a $L^1 \to L^{p}$ gain of integrability for solutions to (\ref{P1}). The proof follows almost directly from the nonlineal Poincare's inequality 
\begin{align} \label{GKSineq}
\int_{\mathbb{R}^d} u^{p+1}\;dx \le \left( \frac{p+1}{p}\right)^2 \int_{\mathbb{R}^d} A[u] |\nabla u^{p/2}|^2\;dx,
\end{align}
first proved in \cite{GreKriStr2012}. The gain of integrability we obtain is much faster than the one of the solution to the heat equation, which is of the order of 
$ \frac{1}{t^{\frac{d}{2}\left(1-\frac{1}{p}\right)}}$.

\begin{theorem}\label{Lp thm}
Let $u(t,x)$ be a solution to (\ref{P1}). For any $p>1$ and for all $T>0$, we have
\begin{align*}
    \|u\|_{L^{\infty}(t, \;T,\; L^{p}(\mathbb{R}{^d}))}\leq \frac{c}{t^{1-\frac{1}{p}}},
\end{align*}
with $c$ a constant depending only on $p$ and $\|u_{\text{in}}\|_{L^{1}(\mathbb{R}^{d})}$.
\end{theorem}

\begin{proof} Multiply  (\ref{P1}) by $\varphi:=u^{p-1}$ and integrate the resulting equation in $\mathbb{R}{^d}$.  Integration by parts yields
\begin{equation*}
   \partial_{t}\int u^{p}\;dx=-\frac{4(p-1)}{p}\int \left\langle A[u]\nabla u^{\frac{p}{2}},\nabla u^{\frac{p}{2}}\right\rangle\;dx.
\end{equation*}
Inequality (\ref{GKSineq}) implies 
\begin{equation}\label{ode ineq}
\partial_{t}\int u^{p}(x)\;dx\leq -\frac{4p(p-1)}{(p+1)^{2}}\int u^{p+1}(x)\;dx.
\end{equation}
Combining the interpolation inequality 
\begin{equation*}
    \|u\|_{L^{p}}\leq \|u\|^{\theta}_{L^{1}}\|u\|^{1-\theta}_{L^{p+1}}, \quad  \theta=\frac{1}{p^{2}},
\end{equation*}
with (\ref{ode ineq}) yields
\begin{equation*}
    \partial_{t}\|u\|^{p}_{L^{p}}\leq -C\|u\|^{\frac{p^{2}}{(p-1)}}_{L^{p}},
\end{equation*}
with $C=\frac{4p(p-1)}{(p+1)^{2}} \|u_{\text{in}}\|^{-\frac{1}{p-1}}_{L^{1}}$. Note that the $L^{1}$-norm is conserved.
Define $y:=\|u\|_{L^{p}}^{p}$;  the solution to the differential inequality
\begin{equation*}
    y^{\prime}\leq -C y^{\frac{p}{p-1}},
\end{equation*}
has the bound
\begin{equation*}
    y \le \frac{1}{\left(y_{0}^{-\frac{1}{p-1}}+\frac{C}{p-1}t\right)^{p-1}}.
\end{equation*}
This implies that 
\begin{align*}
    \|u\|_{L^{\infty}(t,
     \;T,\;L^{p}(\mathbb{R}^{d}))}&\leq \left(\frac{(p-1)}{C}\right)^{1-\frac{1}{p}} \frac{1}{t^{1-\frac{1}{p}}}, 
\end{align*}
and this finishes the proof. 
\end{proof}

We also have the following moment estimate: 

\begin{lemma}\label{Lemma 4.1}
 Let $u(t, x)$ be a smooth solution to  (\ref{P1}) in the time interval $[0,T]$ with initial data $u_{\text{in}}\in L^{1}_{m}(\mathbb{R}^{d})$ for some $m\geq 2$. Then there exists a constant $c$ that only depends on $T$ and the $L^{1}_{m}$-norm of  $u_{\text{in}}$ such that
\begin{align*}
    \sup_{t\in [0, T]}\|u(t, x)\|_{L^{1}_{m}(\mathbb{R}^{d})}\leq c.
\end{align*}
\end{lemma}

\begin{proof}
We start with $m=2$. Testing with  $\phi=(1+|x|^{2})$ and integrating by parts yield
\begin{align*}
    \partial_{t}\int u(1+|x|^{2})dx
    &\le 4\int u\;|x| |\nabla A[u]|\;dx+4d\int u\;\text{Tr}(A[u])\;dx\\
&=:\mathcal{J}_{1}+\mathcal{J}_{2}.
\end{align*}
Let us first estimate $\mathcal{J}_{2}$. 
Applying the first estimate of Lemma \ref{prop 1} to $\mathcal{J}_{2}$, we get 
\begin{align}\label{new estimate}
    \mathcal{J}_{2}&\leq C_{0}\|u\|^{\frac{p(d-2)}{(p-1)d}}_{L^{p}(\mathbb{R}^{d})}\|u\|^{\frac{2p-d}{d(p-1)}+1}_{L^{1}(\mathbb{R}^{d})}.
\end{align}
Then an application of Theorem \ref{Lp thm} to the  $L^{p}$-norm of (\ref{new estimate}), gives
\begin{align}
    \mathcal{J}_{2}\lesssim \frac{1}{t^{1-\frac{2}{d}}}\|u\|_{L^{1}(\mathbb{R}^{d})}^{\frac{2p-d}{d(p-1)}+1}.
\end{align}
Next, we estimate $\mathcal{J}_{1}$. We have
\begin{align*}
    \mathcal{J}_{1}
    & \leq \|\nabla A[u]\|_{L^{\infty}}\int u\;(1+|x|^{2})\;dx.
\end{align*}
Apply once more Lemma \ref{prop 1} and Theorem \ref{Lp thm} to get
\begin{align*}
    \mathcal{J}_{1}& \lesssim \frac{1}{t^{1-\frac{1}{d}}} \int u (1+|x|^{2})\;dx.
\end{align*}
Gathering the estimates of $\mathcal{J}_{1}$ and $\mathcal{J}_{2}$ together, we acquire the bound 
\begin{align*}
    \partial_{t}\int u(1+|x|^{2})dx&=\mathcal{J}_{1}+\mathcal{J}_{2}\\
    &\leq \frac{c}{t^{\frac{d-1}{d}}}\int u (1+|x|^{2})\;dx+\frac{C}{t^\frac{d-2}{d}},
\end{align*}
where $c:=\|u\|^{\frac{p-d}{d(p-1)}}_{L^{1}(\mathbb{R}^{d})}$ and $C:=\|u\|^{\frac{2p-d}{d(p-1)}+1}_{L^{1}}$.
The last inequality is equivalent to the differential inequality 
\begin{align}\label{new ODE}
    y^{\prime}\leq \frac{c}{t^{\frac{d-1}{d}}}y+\frac{C}{t^{\frac{d-2}{d}}},
\end{align}
which, after multiplying by $\mu(s)=e^{-ds^{\frac{1}{d}}}$,  reduces to
\begin{align*}
    (y\mu(s))^{\prime}\leq\mu(s)s^{-\frac{d-2}{d}},
\end{align*}
and has solution:
\begin{align*}
y(t)=e^{dt^{1/d}}\left\{\int_{0}^{t}e^{-ds^{1/d}}s^{\frac{2-d}{d}}ds+y_{0}e^{-dt^{\frac{1}{d}}}\right\}.
\end{align*}
Applying the same argument iteratively, we can get the estimate for any $m>2$.
\end{proof}

\begin{remark}\label{moment bounds}
    Thanks to the bound on the second moments from Lemma \ref{Lemma 4.1}, the conservation of mass and the decay of entropy, the matrix $A[u]$ satisfies the following ellipticity condition:
    \begin{align}\label{entropy bound}
        A[u](x,t)\geq \frac{c(T)}{\langle x \rangle^{d}}\quad \text{for any} \quad x\in \mathbb{R}{^d}, \; t\in[0,T].
    \end{align}
\end{remark}

\section{$L^1 \to L^{\infty}$ gain of integrability}\label{sec:De Giorgi}
In this section we first show the $L^p \to L^{\infty}$ gain of integrability for solutions to (\ref{P1}). This, combined with the estimate of Theorem \ref{Lp thm} will conclude the proof of Theorem \ref{main thm}. We follow a modification of the De Giorgi iteration previously used in \cite{GoGuLo23} and \cite{GoLo23}. We start with a technical lemma. Let $M>0$ and $t>0$; for each $k\in \mathbb{N}$, define 
$$
C_{k}:=M(1-2^{-k}), \quad T_{k}:=\frac{t}{2}\left(1-\frac{1}{2^{k}}\right). 
$$
Note that $M$ is considered constant with respect to $t$ once $t$ has been fixed.

We denote with $(u-c)_{+}$ the maximum between $0$ and $(u-c)$.

\begin{lemma}\label{lemma Ck bound}
Let $p>d/2$, $\gamma>0$ defined as 
\begin{align*}
    \gamma=-1+\frac{2}{d}p-\frac{3}{m}(d-2)(p-1),
\end{align*}
and $m\ge2$ such that 
$$
m>\frac{3d(d-2)}{2} \max\left\{1,\frac{p-1}{p-\frac{d}{2}}\right\}.
$$
For each  $k\geq 1$ we have the bound

    \begin{align*}
    \int_{\mathbb{R}^{d}}(u-C_{k})^{p}_{+}dx
    \leq \left(\frac{c_0 2^{k}}{M}\right)^{1+\gamma}
    &\left(\|\langle \cdot \rangle^{-d/2}\nabla (u-C_{k-1})^{\frac{p}{2}}_{+}\|^{2}_{L^{2}}+\|(u-C_{k-1})_{+}\|^{p}_{L^{p}}\right)\\
    &\cdot \|(u-C_{k-1})_{+}\|^{p\left(\frac{2}{d}-\frac{3}{m}(d-2)\right)}_{L^{p}}\|(u-C_{k-1})_{+}\|^{\frac{3}{m}(d-2)}_{L^{1}_{m}},
\end{align*}
with $c_{0}$ dimensionless constant. 
\end{lemma}

\begin{proof}
Observe that $0\leq C_{k-1}<C_{k}$. From this we have 
\begin{align}\label{Ck estimate}
    0\leq (u-C_{k})_{+}\leq (u-C_{k-1})_{+}.
\end{align}
Moreover $u-C_{k-1}=u-C_{k}+C_{k}-C_{k-1}$. Dividing by $C_{k}-C_{k-1}$ we acquire on the set $\{u\geq C_{k}\}$,
\begin{align*}
    \frac{u-C_{k-1}}{C_{k}-C_{k-1}}=\frac{u-C_{k}}{C_{k}-C_{k-1}}+1\geq 1.
\end{align*}
This tells us that
\begin{align*}
    \mathbbm{1}_{\{u-C_{k}\geq 0\}}\leq \frac{(u-C_{k-1})_{+}}{C_{k}-C_{k-1}}.
\end{align*}
Hence, for any $a>0$ we have 
\begin{align*}
    \mathbbm{1}_{\{u-C_{k}\geq 0\}}\leq \left(\frac{(u-C_{k-1})_{+}}{C_{k}-C_{k-1}}\right)^{a}.
\end{align*}
Multiplying the above inequality by $(u-C_{k})_{+}$ and using (\ref{Ck estimate}), we deduce
\begin{align}\label{u-k ineq}
(u-C_{k})_{+}\leq\frac{(u-C_{k-1})^{1+a}_{+}}{(C_{k}-C_{k-1})^{a}}\quad\text{for any}\quad a>0.
\end{align}
Chose $a=\frac{1+\gamma}{p}$ for some $\gamma>0$ to be defined later. Inequality  (\ref{u-k ineq}) implies 
\begin{align*}
    \int_{\mathbb{R}^{d}}(u-C_{k})^{p}_{+}\;dx&\leq \left(\frac{2^{k}}{M}\right)^{1+\gamma}  \int_{\mathbb{R}^{d}}(u-C_{k-1})^{p+1+\gamma}_{+}\;dx. 
\end{align*}
Lemma \ref{Holde-interpolation ineq} with $q=1+\gamma+p$ yields
\begin{align*}
    \int_{\mathbb{R}^{d}}(u-C_{k})^{p}_{+}\;dx&\leq c_{0}\left(\frac{2^{k}}{M}\right)^{1+\gamma}\left(\left\|\nabla (u-C_{k-1})^{\frac{p}{2}}_{+}\langle\cdot\rangle^{-\frac{d}{2}}\right\|^{2}_{L^{2}}+\|(u-C_{k-1})_{+}\|^{p}_{L^{p}}\right)\\
    &\quad\quad\quad \cdot \|(u-C_{k-1})_{+}\|^{p\left(\frac{\frac{d-2}{d}+\gamma}{p-1}\right)}_{L^{2}}\|(u-C_{k-1})_{+}\|^{\frac{2p-d-d\gamma}{d(p-1)}}_{L^{1}_{m}},
\end{align*}
  with $c_{0}$ dimensionless constant and $m=\frac{3d(d-2)(p-1)}{(d+2)p-d(1+\gamma+p)}$.
  Next, we express $\gamma$ in terms of $m$, and get $\gamma=\frac{2}{d}p-1-\frac{3(d-2)(p-1)}{m}$, which implies, after substitution in the norms,
  \begin{align*}
      \int_{\mathbb{R}^{d}}(u-C_{k})^{p}_{+}\;dx \leq c_0 \left(\frac{2^{k}}{M}\right)^{1+\gamma}
    &\left(\left\|\langle \cdot \rangle^{-d/2}\nabla (u-C_{k-1})^{\frac{p}{2}}_{+}\right\|^{2}_{L^{2}}+\|(u-C_{k-1})_{+}\|^{p}_{L^{p}}\right)\\
    &\cdot \|(u-C_{k-1})_{+}\|^{p\left(\frac{2}{d}-\frac{3}{m}(d-2)\right)}_{L^{p}}\|(u-C_{k-1})_{+}\|^{\frac{3}{m}(d-2)}_{L^{1}_{m}}.
  \end{align*}
The constraint $\gamma>0$ implies $m>\frac{3d(d-2)}{2}\frac{(p-1)}{p-\frac{d}{2}}$. The proof of the lemma is complete after recalling Remark \ref{remark on constraint m}.
\end{proof}

We are now ready to start the De Giorgi iteration. For any $k\ge 1$ let us define the energy $ \mathcal{E}_{k}$ as 
\begin{align*}
    \mathcal{E}_{k}(T_{k+1},t):=\sup_{\tau\in (T_{k+1},t)}\int (u-C_{k})^{p}_{+}(\tau,x)dx+C(p)\int_{T_{k+1}}^{t}\int \langle x \rangle^{-d}\left|\nabla(u-C_{k})^{\frac{p}{2}}_{+}\right|^{2}dxd\tau,
\end{align*}
and  $\mathcal{E}_{0}$ as 
\begin{align}\label{initial energy}
    \mathcal{E}_{0}:=\sup_{(t/4, t)}\int_{\mathbb{R}^{d}} u^{p}\;dx+C(p)\int_{t/4}^{t}\int_{\mathbb{R}^{d}}\langle x\rangle^{-d}\left|\nabla u^{\frac{p}{2}}\right|^{2}\;dx\;d\tau, \quad C(p):=\frac{4(p-1)}{p}.
\end{align}

\begin{lemma}\label{lemma on De Giorgi method}
    Given $p>\frac{d}{2}$, $\gamma=-1+\frac{2}{d}p-\frac{3(d-2)(p-1)}{m}$ and $m>\frac{3d(d-2)}{2} \max\left\{1,\frac{p-1}{p-\frac{d}{2}}\right\}$.  For all $k\geq 1$ we have 
    \begin{align*}
        \mathcal{E}_{k}(T_{k+1},t) \lesssim\frac{1}{tM^{1+\gamma}}\;\mathcal{E}_{k-1}(T_{k},t)^{\left(1+\frac{2}{d}-\frac{3}{m}(d-2)\right)}.    \end{align*}
\end{lemma}

\begin{proof}
We test  (\ref{P1}) with $(u-C_k)_+^{p-1}$, integrate in $\mathbb{R}^d \times (s,\tau)$ with 
$0\leq T_{k}\leq s\leq T_{k+1}\leq \tau $. After averaging on $s$ between $T_{k}$ and $T_{k+1}$, and taking the supremum of $\tau$ in $(T_{k+1}, t)$ we get 
\begin{align*}
\sup_{\tau\in (T_{k+1},t)}\int (u-C_{k})^{p}_{+}(\tau,x)dx&+C(p)\int_{T_{k+1}}^{t}\int A[u]\left|\nabla(u-C_{k})^{\frac{p}{2}}_{+}\right|^{2}dxds\\
        &\leq \frac{1}{T_{k+1}-T_{k}}\int_{T_{k}}^{t}\int(u-C_{k})^{p}_{+}dx\;ds,
\end{align*}
which can be also written as 
\begin{align}\label{energy bound}
\mathcal{E}_{k}(T_{k+1},T)\leq \frac{1}{T_{k+1}-T_{k}}\int_{T_{k}}^{t}\int(u-C_{k})^{p}_{+}dx\;ds.
\end{align}
Since $(T_{k+1}-T_{k})=\frac{t}{2^{k+2}}$, we apply the integral bound of Lemma \ref{lemma Ck bound} to get 
\begin{align*}
&\mathcal{E}_{k} \lesssim\frac{2^{k+1}}{t }\left(\frac{2^{k}}{M}\right)^{1+\gamma}\sup_{(T_k,t)}\|(u-C_{k-1})_{+}\|_{L^{1}_{m}}^{\frac{3(d-2)}{m}}\sup_{(T_k,t)}\|(u-C_{k-1})_{+}\|_{L^{p}}^{p\left(\frac{2}{d}-\frac{3}{m}(d-2)\right)}\\
&\quad\cdot\left[\sup_{(T_k,t)}\|(u-C_{k-1})_{+}\|_{L^{p}}^{p}+\int_{T_{k}}^{t}\|\langle\cdot\rangle^{-d/2}\nabla(u-C_{k-1})_{+}^{\frac{p}{2}}\|_{L^{2}}^{2}\;ds\right]\\
&\le \frac{2^{k+1}}{t }\left(\frac{2^{k}}{M}\right)^{1+\gamma}\sup_{(T_k,t)}\|(u-C_{k-1})_{+}\|_{L^{1}_{m}}^{\frac{3(d-2)}{m}} \\ 
&\quad\cdot\left[\sup_{(T_k,t)}\|(u-C_{k-1})_{+}\|_{L^{p}}^{p}+\int_{T_{k}}^{t}\|\langle\cdot\rangle^{-d/2}\nabla(u-C_{k-1})_{+}^{\frac{p}{2}}\|_{L^{2}}^{2}\;ds\right]^{1+\frac{2}{d}-\frac{3}{m}(d-2)}\\
& =\frac{C^{k}C_{0}}{tM^{1+\gamma}}\mathcal{E}_{k-1}^{\left(1+\frac{2}{d}-\frac{3}{m}(d-2)\right)},
\end{align*}
with $C_{0}:=\sup_{(0,T)}\|u\|_{L^{1}_{m}}^{\frac{3}{m}(d-2)}$. 
\end{proof}

For simplicity in the notation, we define $\beta_{1}:=\frac{2}{d}-\frac{3}{m}(d-2)$. The inequality of the previous lemma shows that, iteratively,  
\begin{align}\label{ineq on c/T}
\mathcal{E}_{k}&\lesssim\left( \frac{c_0}{t^{\frac{1}{\beta_{1}}} M^{\frac{(1+\gamma)}{\beta_1}}} \mathcal{E}_{0}\right)^{(1+\beta_{1})^{k}}.
\end{align}
Recall the definition of $\mathcal{E}_{0}$:
\begin{align*}
\mathcal{E}_{0}= \sup_{(t/4,t)}\int u^{p}(s,x)\;dx+C(p)\int_{t/4}^{t}\int A[u] \left|\nabla u^{\frac{p}{2}}\right|^{2}\;dx\;ds, \quad C(p):=\frac{4(p-1)}{p}.
\end{align*}
Since
$$
\mathcal{E}_{0} \le c\sup_{(t/4, \infty)} \int u^{p}(s,x)\;dx, 
$$
where $c$ is a dimensionless constant greater than $2$.
 Theorem \ref{Lp thm} implies
 \begin{align*}
\mathcal{E}_{0}\leq \frac{c_{p}}{t^{p-1}},
\end{align*}
where $c_{p}\approx (p-1)^{p-1}$ only depends on $p$ and on the $L^1$-norm of the initial data.  Passing to the limit $k\to +\infty$ in (\ref{ineq on c/T}) we obtain 
$$
u\le M,
$$
provided 
$$
M  \gtrsim \frac{c_p}{t^{1+\varepsilon}}, 
$$
with $\varepsilon = \frac{1-\frac{2}{d}}{1+\gamma}$ and $c_p$ only dependent on the $L^1_m$-norm of the initial data and $p>1$ . Note that $\varepsilon >0$ can be as small as one wishes by choosing $p$ arbitrarily large. To see this, first note that if $p$ is greater than $d-1$ then $ 1< \max\left\{1,\frac{p-1}{p-\frac{d}{2}}\right\} \le 2$. Then, taking $m > 3d(d-2)$, we get 
\begin{align*}
\varepsilon = \frac{1-\frac{2}{d}}{1+\gamma} = \frac{1-\frac{2}{d}}{\frac{2p}{d}- \frac{3(d-2)(p-1)}{m}}\le \frac{d-2}{p+1}.
\end{align*}
This finishes the proof of Theorem \ref{main thm}.

{{
\bibliography{Landau_Diffusion}
\bibliographystyle{plain}
}}

\end{document}